\documentclass{imanum}

\jno{drnxxx}
\received{December 3, 2015}

\theoremstyle{theorem}
\newtheorem{assumption}{\sc Assumption}[section]
\newtheorem{definition_}{\sc Definition}[section]
\newtheorem{example_}{\sc Example}[section]

\include{erik_header}
\usepackage{enumitem}
\usepackage{stackengine} 

\newcommand{\CCC}{\mathcal{C}}

\newcommand{\uj}{u(t_j)}
\newcommand{\ujm}{u(t_{j-1})}
\newcommand{\sumsplit}{\frac1q\sum_{k=1}^q{\alpha_k}}

\DeclareMathOperator*{\essinf}{ess\,inf}

\newcommand{\tr}{\gamma}
\DeclareMathOperator*{\supp}{supp}

\title{Additive domain decomposition operator splittings -- convergence analyses in a dissipative framework}
\shorttitle{Additive domain decomposition operator splittings}

\author{%
{\sc
Eskil Hansen\thanks{Email: eskil@maths.lth.se}
and
Erik Henningsson\thanks{Corresponding author. Email: erikh@maths.lth.se}
} \\[2pt]
Centre for Mathematical Sciences, Lund University, PO Box 118, SE-221 00 Lund, Sweden}
\shortauthorlist{E. Hansen and E. Henningsson}

\begin{document}
\maketitle
\thispagestyle{empty}

\begin{abstract}
{We analyze temporal approximation schemes based on overlapping domain decompositions. As such schemes enable computations on parallel and distributed hardware, they are commonly used when integrating large-scale parabolic systems. Our analysis is conducted by first casting the domain decomposition procedure into a variational framework based on weighted Sobolev spaces. The time integration of a parabolic system can then be interpreted as an operator splitting scheme applied to an abstract evolution equation governed by a maximal dissipative vector field. By utilizing this abstract setting, we derive an optimal temporal error analysis for the two most common choices of domain decomposition based integrators. Namely, alternating direction implicit schemes and additive splitting schemes of first and second order. For the standard first-order additive splitting scheme we also extend the error analysis to semilinear evolution equations, which may only have mild solutions.}
{domain decomposition; convergence order; additive splitting schemes; alternating direction implicit schemes; parabolic equations; semilinear evolution equations.}
\end{abstract}

\section{Introduction}


The numerical solution of partial differential equations (PDEs) often require large-scale computations on parallel and distributed computing systems. Domain decomposition methods resolve the issue by separating the computational domain into a family of subdomains. Then, an iteration is usually carried out where, in every step, the PDE is solved on each subdomain after which each such solution is communicated to the neighbouring subdomains (only). Due to the independence of the decomposed PDEs and the lack of global communication domain decomposition methods are ideal for parallel and distributed computing. See \cite{Mathew2008,Quarteroni1999} and \cite{Toselli2004} for general surveys.

When domain decomposition methods are used in the context of parabolic PDEs the classical approach is to first approximate the PDE by a standard implicit time discretization method. Then, for the series of resulting stationary problems, a domain decomposition method is applied at each time level, either as an iterative stationary equation solver or as a preconditioner for an iterative solver.

In this paper, however, we consider domain decomposition operator splittings (DDOSs), also know as regionally-additive schemes, which were introduced by, e.g., \cite{Vabishchevich1989}. In contrast to classical domain decompositions a DDOS do not require iterations at each time step, i.e.\ a reduction in the computational cost of the overall solution procedure may be achieved. To construct a DDOS consider a PDE defined on a domain $\Omega$. Furthermore, consider a family of overlapping subdomains $\{\Omega_k\}_{k=1}^q$ of $\Omega$ and let $\delta$ denote the characteristic length of the overlaps. Define on this family a partition of unity $\{\chi_k\}_{k=1}^q$ such that each function $\chi_k$ is continuous and vanishes except on the corresponding subdomain $\Omega_k$. These weight functions are then used to formulate the domain decomposition as an operator splitting where each operator represents the vector field of the PDE on the corresponding subdomain. The discretization in time is then given by employing any of the great number of splitting methods available in the literature.    

Generally, splitting methods are time stepping schemes for the temporal discretization of evolution equations
\begin{equation}
\dot u = Au = \sum_{k=1}^q A_ku, \quad u(0) = \eta,
\label{eq:evolution_equation}
\end{equation}
for some (possibly unbounded) operators $A_1,\dots,A_q$. These methods may be efficient whenever the subproblems $\dot u = A_k u$ are less costly to solve (numerically or exactly) than the original problem $\dot u = Au$. For DDOSs the numerical solution of any subproblem entails approximately solving the original PDE on a subdomain only, therefore the formulation of domain decompositions as operator splittings is a promising basis for efficient schemes. For general surveys on splitting methods we refer to \cite{Hundsdorfer,Marchuk,McLachlan} and \cite{Hairer}.

We will pay extra attention to two classes of splitting methods: the alternating direction implicit (ADI) schemes and the additive splitting schemes (also known as sum splittings or summarized approximation schemes). The former since they have been widely used in the context of DDOSs and the latter since they are extra suitable for DDOSs implemented on parallel hardware. 

The ADI methods were first introduced by \cite{Douglas1955} and \cite{Peaceman}. An overview is given by \cite{Hundsdorfer}, see also \cite{Hansen2013,Hundsdorfer1989} and \cite{Lions} for convergence studies. These schemes are most often used for two-operator splittings, i.e., $q = 2$, where the first-order Douglas--Rachford scheme
\begin{equation}
S_h = (I - h A_2)\inv (I - h A_1)\inv (I + h^2 A_1 A_2),
\label{eq:DR}
\end{equation}
is one of the most well-known. Here $S_h$ denotes the numerical flow of the scheme, i.e., a time step of size $h$ is given by applying $S_h$. The ADI schemes were first used in the context of DDOSs by \cite{Vabishchevich1989} and they have since been extensively studied in this context, see e.g.\ \cite{Arraras2015,Mathew1998,Vabishchevich2008} and the references therein. When the scheme \eqref{eq:DR} is applied to parabolic PDEs it exhibits excellent stability properties and a favorable local error structure. In the context of DDOSs, when assuming that the functions $\chi_k$ and the solution $u$ are smooth enough, the leading error term of the scheme \eqref{eq:DR} is of size $\Ordo(h/\delta)$. This should be compared with the Lie fractional step splitting method which has a leading local error term of size $\Ordo(h/\delta^3)$, cf.\ \cite{Hansen2013} and \cite{Mathew1998}. However, ADI schemes with more than two operators often exhibit stability issues when applied to parabolic PDEs, especially for advection dominated problems, cf.\ \cite{Douglas1964} and \cite[Section~IV.3.2]{Hundsdorfer}. 
In the presence of a spatial discretization this implies a CFL time step restriction. Furthermore, the required regularity of the solution $u$ (and the weights $\chi_k$ in the context of DDOSs) increases with the number of operators, cf.\ \cite[Theorem~4.5]{Mathew1998}.

On the other hand the first-order additive splitting scheme
\begin{equation}
S_h = \frac1q \sum_{k=1}^q (I-hqA_k)\inv
\label{eq:sum_splitting}
\end{equation}
suffers from none of these adverse effects; We will see that this scheme is stable for parabolic PDEs independently of the number of operators $q$ considered and that the structure of the truncation error is also independent of $q$. Furthermore, additive schemes have a great potential for efficient implementation on parallel hardware as the actions of the resolvents $(I-hqA_k)\inv$ can be computed in parallel. Finally the scheme can easily be extended to semilinear PDEs. As a higher-order alternative to the scheme \eqref{eq:sum_splitting} we will also analyze a second-order additive splitting scheme based on the fractional step Crank--Nicolson method, cf.\ \cite[Section~IV.2.2]{Hundsdorfer} and \cite{Swayne1987}. Early references to additive splitting schemes include \cite{Coron1982} and \cite{Gordeziani1974}. For contemporary studies in the context of DDOSs we refer to \cite{Samarskii2013,Vabishchevich2013} and the references therein.

For completeness we mention also the widely used exponential splitting schemes.
These methods are good choices whenever the exact solution of each subproblem $\dot u = A_k u$ can be easily computed. However, this is generally not the case for DDOSs due to the presence of the weight functions $\chi_k$. Thus, we will not discuss exponential splitting schemes any further.

Splitting schemes for parabolic PDEs are often analyzed in the abstract framework of maximal dissipative operators. Due to the abstract setting such analysis results apply not only to parabolic PDEs, but rather to a wide range of evolution equations. Furthermore, these results are independent of spatial discretizations, i.e.\ if a space discretization is subsequently applied there will be no CFL time step restrictions. In the literature there are some partial analyses regarding splitting schemes for dissipative evolution equations in the context of DDOSs. These are, however, performed in finite dimensional spaces with linear bounded operators given after space discretizations. 
Convergence order results with constants depending on $\norm{A_k}$ -- which implies CFL conditions -- are given by \cite{Samarskii2013}. Further stability studies for additive splitting schemes and for fractional step methods are carried out by \cite{Vabishchevich2013} and \cite{Portero2010}, respectively. In the latter the vector fields may be nonlinear but are assumed to be continuously differentiable.


However, to the best of our knowledge, DDOSs of parabolic PDEs have still not been analyzed without first applying a space discretization. Thus, the first goal of this paper is the formulation of DDOSs in the framework of maximal dissipative operators using weighted Sobolev spaces. The second goal is to perform optimal convergence order analyses for the aforementioned ADI schemes and additive splitting schemes in this framework. 
Furthermore, for the additive scheme \eqref{eq:sum_splitting} we will also extend the analysis to semilinear evolution equations, which may only have mild solutions. Two semilinear example applications are given in the context of DDOSs: parabolic problems with nonlinear reaction terms and a phase-field model. Finally, we emphasize that, due to the generality of the dissipative framework, our error analysis applies beyond parabolic PDEs and DDOSs.


\section{Domain decomposition operator splittings} \label{sec:domain_decomposition}

In this section we demonstrate how to formulate domain decomposition operator splittings in the framework of maximal dissipative operators using a linear diffusion-advection-reaction equation as a model problem. That is, we will formulate the DDOS of this parabolic PDE as a dissipative evolution equation of the type \eqref{eq:evolution_equation}.


Let $d$ be a positive integer and let $\Omega \subset \R^d$ be an open and bounded Lipschitz domain such that $\partial\Omega \in \CC^2$.
(See also Remark~\ref{rem:convex_domain} on domains with non-smooth boundaries.) 
Consider the diffusion-advection-reaction equation
\begin{equation} 
\left\{ \begin{aligned}
&\dot u = Au = \nabla \cdot (\lambda \nabla u) - \rho \cdot \nabla u - \sigma u, \quad &&\text{in } \Omega\times (0,T], \\
&Bu = 0, \quad &&\text{on } \partial\Omega\times (0,T], \\
&u(\cdot,0) = \eta, \quad &&\text{in } \Omega,
\end{aligned} \right.
\label{eq:rda}
\end{equation}
where the scalar function $\sigma$ and the elements of the $d$-vector $\rho$ and the $d\times d$-matrix $\lambda$ fulfill the following regularity assumptions: $\sigma, \rho_i \in L^\infty(\Omega)$ and $\lambda_{i,j} \in \CC^1(\widebar\Omega)$ for all $i,j = 1,\dots, d$.
Furthermore, we assume that the uniform ellipticity condition is fulfilled, i.e.\ there is a constant $\lambda_0 > 0$ such that
\[ \sum_{i,j = 1}^d {\lambda_{i,j}(x) \xi_i \xi_{j}} \geq \lambda_0|\xi|^2, \quad \text{for all } x \in \Omega \text{ and } \xi \in \R^d, \]
where $|\cdot|$ denotes the Euclidean norm in $\R^d$. We consider either homogeneous Dirichlet or Neumann problems. In the former case $B$ is the identity meaning $Bu = u = 0$ on $\partial\Omega \times (0,T]$. In the latter case we have
\[ Bu = (\lambda \nabla u) \cdot n = 0, \quad \text{on } \partial\Omega \times (0,T], \]
where $n$ is the unit outer normal on $\partial\Omega$. 

\begin{remark}
We do not need to assume any bounds on $\rho$ and $\sigma$. Instead, as we shall see in Lemma~\ref{lem:max_diss} and the theorems of Section~\ref{sec:abstract_convergence}, any badly behaved advections or reactions will end up as time step restrictions. However, these restrictions are generally not problematic since they are independent of domain decompositions and possible space discretizations.
\end{remark}

\subsection{Partition of unity} \label{sec:partition_of_unity}

Assume that the domain $\Omega$ is the union of a collection $\{\Omega_k\}_{k=1}^q$ of open subdomains where any pair of subdomains may overlap. Assume further that there is a partition of unity $\{\chi_k\}_{k=1}^q$ subordinate to this cover such that for all $k = 1,\dots, q$, we have
\begin{equation}
\chi_k \in W^{1,\infty}(\Omega), \quad 0 \leq \chi_k \leq 1, \quad \sum_{k=1}^q \chi_k = 1 \quad \text{and} \quad \supp(\chi_k) \subset \widebar\Omega_k.
\label{eq:chi_k}
\end{equation}
Since $W^{1,\infty}(\Omega) \subset \CC(\widebar\Omega)$ all relations in \eqref{eq:chi_k} hold pointwise, cf.\ \cite[Theorem~5.8.4]{Evans2010}. 
The domain decomposition operator splitting is then formally given by
\begin{equation}
Av = \sum_{k=1}^q A_k v, \quad \text{where} \quad A_k v = \nabla \cdot (\chi_k \lambda \nabla v) - \chi_k \rho \cdot \nabla v - \chi_k \sigma v,
\label{eq:A_k_PoU_pre_def}
\end{equation}
for $k = 1,\dots, q$. In Section~\ref{sec:analytic_setting} we give proper variational definitions for $A_1, \dots, A_q$ and $A$.
\begin{figure}[!t]
\centering
\subfigure{
		\stackinset{l}{0mm}{t}{0mm}{\bfseries(a)}{\includegraphics[height=.147\paperheight]{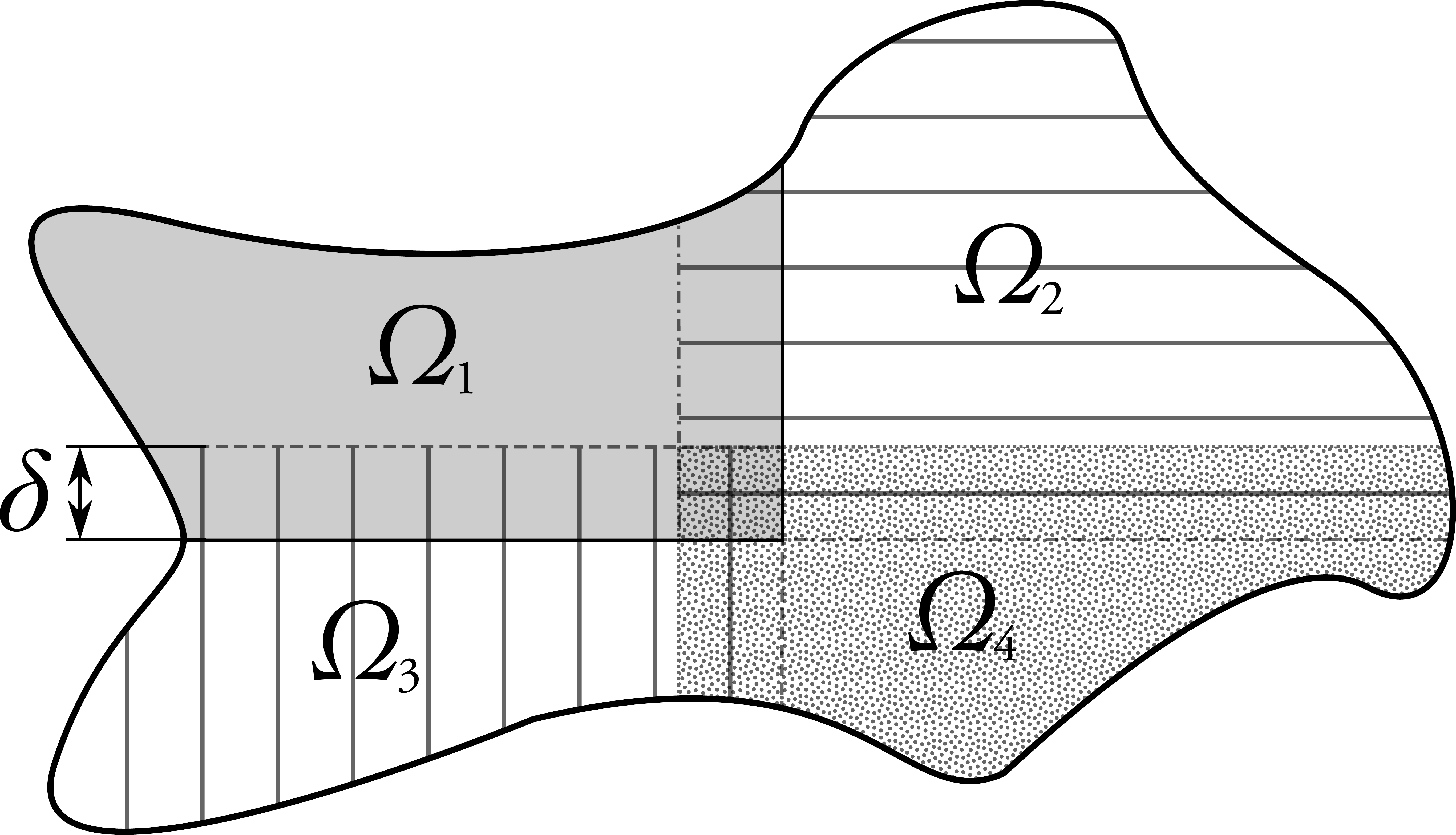}}}
\subfigure{
		\stackinset{l}{0mm}{t}{0mm}{\bfseries(b)}{\includegraphics[height=.147\paperheight]{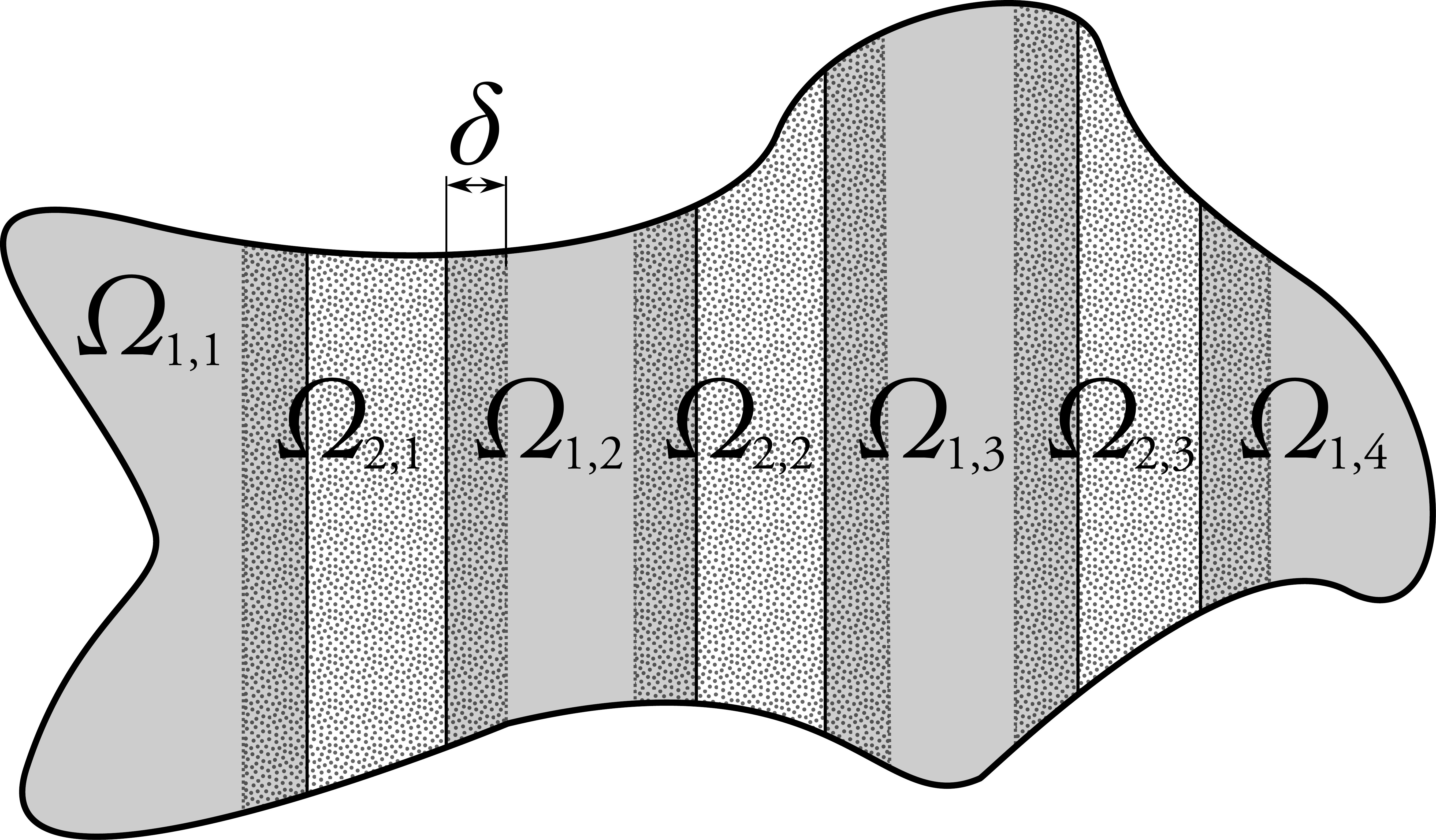}}}
\subfigure{
		\stackinset{l}{0mm}{t}{0mm}{\bfseries(c)}{\includegraphics[height=.147\paperheight]{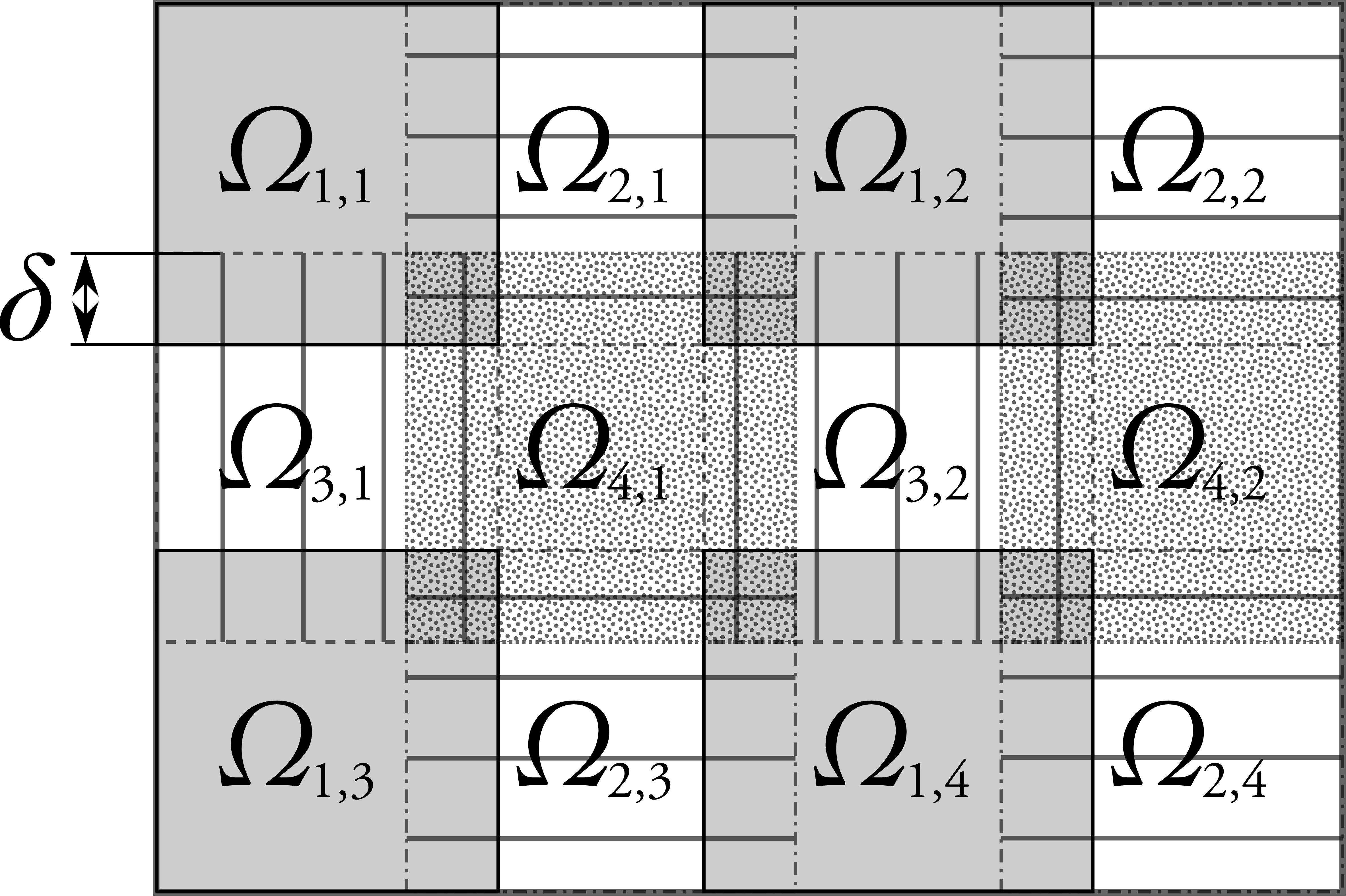}}}
\caption{The figure shows three examples of overlapping decompositions. In subfigure (a) a decomposition into four subdomains is given on a domain with a smooth boundary. This means that the resulting operator splitting will consist of four operators $A_1, \dots, A_4$ defined by the four weight functions $\chi_1, \dots, \chi_4$. Whereas the additive splitting scheme \eqref{eq:sum_splitting} can treat any number of operators, the Douglas--Rachford scheme \eqref{eq:DR} is only stable for two operators. Thus, we also give an example decomposition with two subdomains, $\Omega_1$ and $\Omega_2$, in subfigure (b). To achieve an efficient domain decomposition each of the two subdomains is constructed as the union of a family of disjoint sets, $\Omega_k = \bigcup_\ell \Omega_{k,\ell}$. Here, this is done by an overlapping stripes decomposition. Note that, since $\Omega_{1,\ell}$ and $\Omega_{1,\ell'}$ are disjoint for all $\ell \neq \ell'$, computations of the action of the resolvent $(I-hA_1)\inv$ can easily be parallelized. Of course the same holds for the operator $A_2$. Finally, subfigure (c) shows an example of a decomposition with four subdomains where each subdomain is the union of a family of disjoint sets. This is a widely used decomposition on rectangular domains. Note that more operators grant a greater flexibility which is even more essential for problems with higher dimensionality. For details we refer to \cite[Section~2.5.1]{Toselli2004} and \cite[Section~1.5.6]{Quarteroni1999}.}%
\label{fig:domain_decomposition}%
\end{figure}

We refer to \cite[Section~3.2]{Arraras2015} and \cite[Section~4.1]{Mathew1998} for details on the construction of domain decompositions $\{\Omega_k\}_{k=1}^q$ and partitions of unity $\{\chi_k\}_{k=1}^q$. Of particular interest is \cite[p.~920]{Mathew1998} which outlines a simple procedure for constructing piecewise smooth partitions that fulfill all the requirements of \eqref{eq:chi_k}.
See also Proposition~4.1 in the same reference asserting the existence of infinitely differentiable partitions of unity for which the $r$th-order derivatives of each $\chi_k$ is of size $\Ordo(1/\delta^r)$, where $\delta$ is the characteristic length of the overlaps.

In Figure~\ref{fig:domain_decomposition} we give three example domain decompositions, which are crucial for the splitting schemes considered in this paper.

\subsection{Analytic setting} \label{sec:analytic_setting}

In this subsection we properly define the linear operators $A_1, \dots, A_q$ and $A$ and prove that they are maximal dissipative. To this end we first define the framework of maximal dissipativity. We do this for general Hilbert spaces and operators since we in Section~\ref{sec:abstract_convergence} will consider a more general setting than domain decomposition operator splittings of the linear parabolic PDE \eqref{eq:rda}.
\begin{definition_} \label{def:maximal_dissipative}
Let $\HH$ be a real Hilbert space with inner product $(\cdot,\cdot)$ and induced norm $\norm{\cdot}$. A (possibly) nonlinear operator $G: \DD(G) \subset \HH \rightarrow \HH$ is maximal (shift) dissipative on $\HH$ if there exists a constant $M[G] \geq 0$ such that $G$ fulfills the dissipativity condition
\begin{equation}
(Gv - Gw, v-w) \leq M[G] \norm{v-w}^2, \quad \text{for all } v,w \in \DD(G),
\label{eq:dissipativity}
\end{equation}
and the range condition
\[ \RR(I-hG) = \HH, \quad \text{for all } h > 0 \text{ such that } hM[G] < 1. \]
\end{definition_}

We will define the linear operator $A$ of \eqref{eq:rda} on a subspace of $H^1_0(\Omega)$ or $H^1(\Omega)$ such that it is maximal dissipative on $L^2(\Omega)$, a construction that is commonly known as the Friedrichs extension, cf.\ \cite[Section~31.4]{Zeidler1990IIB}. However, there is no subspace of $H^1_0(\Omega)$ or $H^1(\Omega)$ on which the operators $A_1, \dots, A_q$ of \eqref{eq:A_k_PoU_pre_def} can be defined such that they are maximal on $L^2(\Omega)$. Instead we will consider weighted Sobolev spaces, see \cite{Kufner1980} and \cite[Section~3.2]{Triebel1978} for an introductory reading.



For each $k = 1,\dots, q$ we define the inner products
\[ (\cdot, \cdot)_{V_k} = (\cdot, \cdot)_{L^2(\Omega)} + (\chi_k\nabla\cdot, \nabla\cdot)_{L^2(\Omega)} \]
on $\CC^1(\widebar\Omega)$, where $(\cdot, \cdot)_{L^2(\Omega)}$ denotes the inner product on $L^2(\Omega)$. Then, we let $H^1_0(\Omega;\chi_k)$ respectively $H^1(\Omega;\chi_k)$ be the completion of $\CC_0^\infty(\Omega)$ respectively $\CC^1(\widebar\Omega)$ in the norm
\[ \norm{\cdot}_{V_k} = \sqrt{(\cdot,\cdot)_{V_k}}. \]
The spaces $H^1_0(\Omega;\chi_k)$ are $H^1(\Omega;\chi_k)$ are complete by construction and it is straightforward to show that
\[ H^1_0(\Omega) \hookrightarrow H^1_0(\Omega; \chi_k) \quad \text{and} \quad H^1(\Omega) \hookrightarrow H^1(\Omega; \chi_k), \]
for $k = 1,\dots, q$, cf.\ \cite[Theorem~3.22]{Adams}. If the PDE \eqref{eq:rda} is equipped with Dirichlet boundary conditions we let $V = H^1_0(\Omega)$, $\CCC = \CC_0^\infty(\Omega)$ and $V_k = H^1_0(\Omega; \chi_k)$, for $k = 1,\dots,q$. If the equation has Neumann conditions we instead let $V = H^1(\Omega)$, $\CCC = \CC^1(\widebar\Omega)$ and $V_k = H^1(\Omega; \chi_k)$. By construction $\CCC$ is dense in each $V_k$ (and in $V$). Further, by identifying $L^2(\Omega)$ with its continuous dual space we get the following chain of imbeddings
\[ V \hookrightarrow  V_k \hookrightarrow L^2(\Omega) \hookrightarrow V_k^* \hookrightarrow  V^*, \]
for $k = 1,\dots, q$, where $V^*$ and $V_k^*$ denotes the continuous dual spaces of $V$ and $V_k$, respectively.


Now, for $k = 1,\dots, q$, define on $\CCC$ the bilinear forms
\[ \begin{aligned}
&b(v,w) = \int_\Omega (\lambda \nabla v) \cdot \nabla w + (\rho \cdot \nabla v) w + \sigma vw \dx \quad \text{and} \\ 
&b_k(v,w) = \int_\Omega (\chi_k \lambda\nabla v) \cdot \nabla w + (\chi_k\rho \cdot \nabla v) w + \chi_k\sigma vw \dx, 
\end{aligned} \]
which are bounded in $\norm{\cdot}_V$ respectively $\norm{\cdot}_{V_k}$, cf.\ \cite[Theorems 7.2.2 and 7.4.1]{Hackbusch}. Then, $b$ and $b_k$ can be uniquely extended to bounded bilinear forms on $V$ and $V_k$, respectively.
Moreover, by first letting $v$ be in $\CCC$ and then using the density of $\CCC$ in $V_k$ and the continuity of the forms $b_k$ we get
\begin{equation}
b_k(v,v) \geq \left(\sigma_0 - \frac{P^2}{2\lambda_0}\right) (\chi_kv,v)_{L^2(\Omega)} + \frac{\lambda_0}{2} (\chi_k\nabla v,\nabla v)_{L^2(\Omega)}, \quad \text{for all } v \in V_k, \text{ for } k = 1,\dots, q,
\label{eq:a_shift_coercive_a_k}
\end{equation}
by a simple modifications of the proof of \cite[Theorem~5.6.8]{Brenner}. Here we use the notation $P^2 = \sum_{i=1}^d\norm{\rho_i}_{L^\infty(\Omega)}^2$ and $\sigma_0 = \essinf_{x\in \Omega} \sigma(x)$. A similar inequality, but without any weight functions $\chi_k$, holds for the form $b$ as a direct consequence of the aforementioned theorem.

From the bilinear forms we define the energetic operators
\[ \hat A: V \rightarrow V^*, \quad \hat Av = -b(v,\cdot) \quad\text{and}\quad \hat A_k: V_k \rightarrow V_k^*, \quad \hat A_kv = -b_k(v,\cdot), \]
for $k = 1,\dots, q$.
Then, for any $h > 0$ such that $h (P^2/(2\lambda_0) - \sigma_0) < 1$, we get from \eqref{eq:a_shift_coercive_a_k} the lower bounds
\[ \begin{aligned}
&\fp{(I - h\hat A_k)v}{v} = \norm{v}_{L^2(\Omega)}^2 + h b_k(v,v) \\
	&\quad \geq \norm{v}_{L^2(\Omega)}^2 + h\left(\sigma_0-\frac{P^2}{2\lambda_0}\right) (\chi_kv,v)_{L^2(\Omega)} + h\frac{\lambda_0}{2} (\chi_k\nabla v,\nabla v)_{L^2(\Omega)} \geq C\norm{v}_{V_k}^2, \quad \text{for all } v \in V_k,
\end{aligned} \]
for $k = 1,\dots, q$, where $C$ is some positive constant. Similarly, $\fp{(I - h\hat A)v}{v}$ is bounded from below by $C\norm{v}_{V}^2$ for all $v$ in $V$. Thus, for all $h > 0$ such that $h (P^2/(2\lambda_0) - \sigma_0) < 1$ the operators $I-h\hat A_1, \dots, I-h\hat A_q$ and $I-h\hat A$ are coercive and since they are also bounded they are all bijective according to Lax--Milgram's lemma.

Finally we define the the Friedrichs extensions $A:\DD(A) \subset L^2(\Omega) \rightarrow L^2(\Omega)$ and $A_k: \DD(A_k) \subset L^2(\Omega) \rightarrow L^2(\Omega)$ as the following restrictions of the corresponding energetic operators
\begin{align}
&\left\{ \begin{aligned}
&\DD(A) = \{v \in V;\,\, \exists\ z \in L^2(\Omega) \text{ with } (z,w) = \fp{\hat Av}{w}\ \forall\ \! w \in V\}, \\
&Av = z,\ \forall\ \! v \in \DD(A),
\end{aligned} \right. \label{eq:A_def} \\
&\left\{ \begin{aligned}
&\DD(A_k) = \{v \in V_k;\,\, \exists\ z \in L^2(\Omega) \text{ with } (z,w) = \fp{\hat A_kv}{w}\ \forall\ \! w \in V_k\},\\
&A_kv = z,\ \forall\ \! v \in \DD(A_k),
\end{aligned} \right. \label{eq:A_k_def} 
\end{align}
for $k = 1,\dots, q$. 
From the construction it is clear that the operators $I-hA_1,\dots, I-hA_q$ and $I-hA$ are bijective when their energetic counter parts are and thus, for $h > 0$ such that $h (P^2/(2\lambda_0) - \sigma_0) < 1$, we have $\RR(I-hA_1) = \dots = \RR(I-hA_q) = \RR(I-hA) = L^2(\Omega)$. Finally, from \eqref{eq:a_shift_coercive_a_k} we get
\[ (A_kv,v) = -b_k(v,v) \leq - \left(\sigma_0 - \frac{P^2}{2\lambda_0}\right) (\chi_kv,v)_{L^2(\Omega)} -\frac{\lambda_0}{2} (\chi_k\nabla v,\nabla v)_{L^2(\Omega)} \leq \left(\frac{P^2}{2\lambda_0} - \sigma_0\right) \norm{v}_{L^2(\Omega)}^2, \]
for all $v \in \DD(A_k)$, for $k = 1,\dots, q$ and similarily for $A$. We have thus proven the following lemma.
\begin{lemma} \label{lem:max_diss}
The linear operators $A_1: \DD(A_1) \subset L^2(\Omega) \rightarrow L^2(\Omega),\, \dots,\, A_q: \DD(A_q) \subset L^2(\Omega) \rightarrow L^2(\Omega)$ and $A:\DD(A) \subset L^2(\Omega) \rightarrow L^2(\Omega)$ defined by \eqref{eq:A_def} and \eqref{eq:A_k_def} are all maximal dissipative on $L^2(\Omega)$ with $M[A_1], \dots, M[A_q]$ and $M[A]$ bounded by $\max\{0,\, P^2/(2\lambda_0) - \sigma_0\}$.
\end{lemma}


The domain $\DD(A)$ in \eqref{eq:A_def} can be further characterized using classical elliptic regularity results. Under the assumed regularity on $\partial\Omega$ and the coefficients $\lambda_{i,j}$, $\rho_i$, and $\sigma$ we have
\begin{equation}
\DD(A) = H^2(\Omega) \cap H^1_0(\Omega)
\label{eq:D(A)_Dirichlet}
\end{equation}
when $V = H^1_0(\Omega)$ and
\begin{equation}
\DD(A) = \{v \in H^2(\Omega);\, \tr(\lambda \nabla v) \cdot n = 0\}
\label{eq:D(A)_Neumann}
\end{equation}
when $V = H^1(\Omega)$. Here $\tr$ denotes the standard linear and continuous trace mapping $H^1(\Omega)$ into $L^2(\partial\Omega)$, see \cite[Section~5.5]{Evans2010} for a definition. The characterization \eqref{eq:D(A)_Dirichlet} is given in \cite[Theorem~9.1.16]{Hackbusch}. For \eqref{eq:D(A)_Neumann} we get from \cite[Theorem~9.1.17]{Hackbusch} that $\DD(A) \subset H^2(\Omega)$ and the trace condition can be proven by an argument similar to those given in \cite[Theorem~1.2]{Temam1979} and \cite[p.~65]{Temam1997}. 

With the above elliptic regularity results we are ready to prove the following lemma asserting that the domain decomposition operator splitting is well-defined.
\begin{lemma} \label{lem:splitting}
The domains defined by \eqref{eq:A_def} and \eqref{eq:A_k_def} fulfill $\DD(A) \subset \bigcap_{k=1}^q \DD(A_k)$ and the equation $Av = \sum_{k=1}^q A_kv$ holds for all $v \in \DD(A)$. 
\end{lemma}
\begin{proof}
Note that, for $k = 1,\dots,q$ and for $v \in H^2(\Omega)$, we have $\chi_k\lambda\nabla v \in (H^1(\Omega))^d$. Thus, by density there is a sequence $(y_\ell)_{\ell=1}^\infty$ such that $y_\ell \in (\CC^1(\widebar\Omega))^d$ for all $\ell$ and $y_\ell$ tends to $\chi_k\lambda\nabla v$ as $\ell$ tends to infinity. Now, for $w \in \CCC$ the divergence theorem gives
\begin{equation}
\int_\Omega y_\ell \cdot \nabla w \dx = - \int_\Omega (\nabla \cdot y_\ell) w \dx + \int_{\partial\Omega} (\tr(y_\ell) \cdot n) \tr(w) \,\mathrm{d}S.
\label{eq:lemma_2_pf_divergence_thm}
\end{equation}

We consider first Neumann boundary conditions and note that due to the continuity of the trace we have that $\tr(y_\ell)$ tends to $\tr(\chi_k\lambda\nabla v)$ as $\ell$ tends to infinity. By using the definition of the trace and the regularity $\chi_k \in W^{1,\infty}(\Omega) \subset H^1(\Omega) \cap \CC(\widebar\Omega)$ we get from a simple density argument that
\[ \tr(\chi_k\lambda\nabla v) \cdot n = \tr(\chi_k)\tr(\lambda\nabla v) \cdot n, \quad \text{for all } v \in H^2(\Omega). \]
Thus, from the domain characterization \eqref{eq:D(A)_Neumann} we get
\[ \int_\Omega (\chi_k\lambda\nabla v) \cdot \nabla w \dx = - \int_\Omega (\nabla \cdot \chi_k\lambda\nabla v) w \dx, \quad \text{for all } v \in \DD(A),\ w \in \CC^1(\widebar\Omega), \]
when letting $\ell$ tend to infinity in \eqref{eq:lemma_2_pf_divergence_thm}. Then, by adding the advection and reaction terms we get
\begin{equation}
\begin{aligned}
b_k(v,w) &= \int_\Omega (\chi_k\lambda \nabla v) \cdot \nabla w + (\chi_k\rho \cdot \nabla v) w + \chi_k\sigma vw \dx \\
	&= - \int_\Omega \left[ \nabla \cdot (\chi_k\lambda \nabla v) - \chi_k\rho \cdot \nabla v - \chi_k\sigma v \right] w \dx, \quad \text{for all } v \in \DD(A),\ w \in \CC^1(\widebar\Omega),
\end{aligned}
\label{eq:lemma_2_pf_bilinear_form_id}
\end{equation}
which holds also for all $w \in H^1(\Omega;\chi_k)$ due to the continuity of the form $b_k$ and the density of $\CC^1(\widebar\Omega)$ in $H^1(\Omega;\chi_k)$.

For Dirichlet conditions the boundary integral in \eqref{eq:lemma_2_pf_divergence_thm} vanishes directly since $\CCC = \CC_0^\infty(\Omega)$. Thus, due to the density of $\CC_0^\infty(\Omega)$ in $H^1_0(\Omega;\chi_k)$ and the continuity of $b_k$, the equality \eqref{eq:lemma_2_pf_bilinear_form_id} holds also for Dirichlet boundary conditions, but with $w \in H^1_0(\Omega;\chi_k)$. That is, in both cases we have that for all $v\in \DD(A)$ there exists a $z \in L^2(\Omega)$, $z = \nabla \cdot (\chi_k\lambda \nabla v) - \chi_k\rho \cdot \nabla v - \chi_k\sigma v$, such that $\fp{\hat A_k v}{w} = -b_k(v,w) = (z,w)$ for all $w \in V_k$, but then $v$ is in $\DD(A_k)$ according to the definition \eqref{eq:A_k_def}. Further, with $A_k v = z$, the equality $Av = \sum_{k=1}^q A_kv$ immediately follows for all $v \in \DD(A)$.

\end{proof}

%

\begin{remark} \label{rem:convex_domain}
When the PDE \eqref{eq:rda} is equipped with homogeneous Dirichlet boundary conditions the Lemmas \ref{lem:max_diss} and \ref{lem:splitting} also hold on convex domains $\Omega$ with possibly non-smooth boundary $\partial\Omega$. We refer to \cite[Corollary~1.2.2.3]{Grisvard} and \cite[Theorem~9.1.22]{Hackbusch} for the necessary modifications.
\end{remark}

\section{Abstract convergence analyses} \label{sec:abstract_convergence}

The current section is dedicated to convergence studies of splitting schemes in the framework of maximal dissipative operators. As a direct consequence of Lemmas \ref{lem:max_diss} and \ref{lem:splitting} all results of this section apply to domain decomposition operator splittings of the diffusion-advection-reaction equation \eqref{eq:rda}. An important conclusion of this section is that the additive splitting scheme \eqref{eq:sum_splitting} is convergent for DDOSs without any further regularity assumptions other than those given on $\Omega$, $\lambda$, $\rho$, $\sigma$ and the weight functions $\chi_k$ in Section~\ref{sec:domain_decomposition}. The same holds for the considered ADI schemes when $q = 2$ and under some weak regularity assumptions on the initial data $\eta$. Furthermore, for regular enough solutions $u$ all considered ADI schemes and additive splitting schemes are either first- or second-order convergent.

In Examples \ref{ex:semilinear_parabolic} and \ref{ex:Caginalp} other domain decomposition applications are considered, which also include nonlinearities. Furthermore, as noted in the introduction, the convergence results of this section are much more general and are not limited to DDOSs, see e.g.\ \cite{Barbu} and \cite{Roubicek2013} for further applications.

%

\subsection{Preliminaries} \label{sec:maximal_dissipative}

Let $\HH$ be a real Hilbert space with inner product $(\cdot,\cdot)$ and induced norm $\norm{\cdot}$ and recall the definition of maximal dissipativity, Definition~\ref{def:maximal_dissipative}. Furthermore, for an operator $G: \DD(G) \subset \HH \rightarrow \HH$ let $L[G] \in [0,\infty]$ be the smallest constant such that
\[ \norm{Gv-Gw} \leq L[G]\norm{v-w}, \quad \text{for all } v,w \in \DD(G). \]
We call $L[G]$ the Lipschitz constant of $G$ and say that $G$ is Lipschitz continuous if $L[G] < \infty$.

A direct consequence of $G$ being maximal dissipative is that, for $h > 0$ such that $hM[G] < 1$, the resolvent $(I-hG)\inv: \HH \rightarrow \DD(G) \subset \HH$ is well-defined and its Lipschitz constant is bounded as
\begin{equation}
L[(I-hG)\inv] \leq 1/(1-hM[G]).
\label{eq:resolvent_bound}
\end{equation}


We will perform convergence analyses of splitting schemes applied to the linear evolution equation \eqref{eq:evolution_equation} where the involved operators fulfill the following two assumptions:
\begin{assumption} \label{as:dissipativity}
The linear operators $A_1:\DD(A_1) \subset \HH \rightarrow \HH,\, \dots,\, A_q:\DD(A_q) \subset \HH \rightarrow \HH$ and $A:\DD(A) \subset \HH \rightarrow \HH$ are maximal dissipative on $\HH$.
\end{assumption}
\begin{assumption} \label{as:splitting}
The domains of the operators fulfill $\DD(A) \subset \bigcap_{k=1}^q \DD(A_k)$ and the equation $A = \sum_{k=1}^q{A_k}$ holds on $\DD(A)$.
\end{assumption}

Note that under Assumption~\ref{as:dissipativity} and for $\eta \in \DD(A)$ the evolution equation \eqref{eq:evolution_equation} has a unique classical solution, i.e.\ $u\in \CC^1([0,T];\HH) \cap \CC([0,T];\DD(A))$, cf.\ \cite[Proposition~II.6.2]{Engel}.

\subsection{ADI schemes}
We first consider the evolution equation \eqref{eq:evolution_equation} split into two linear operators $A = A_1 + A_2$. For the domain decomposition applications we may achieve two-operator splittings adapted for parallel implementation by decomposing the domain into overlapping stripes, cf.\ Figure~\ref{fig:domain_decomposition}b. 

We summarize some convergence results for the Douglas--Rachford scheme \eqref{eq:DR} as well as for the second-order Peaceman--Rachford ADI scheme 
\begin{equation}
S_h = \Big(I - \frac h2 A_2\Big)\inv \Big(I + \frac h2 A_1\Big) \Big(I - \frac h2 A_1\Big)\inv \Big(I + \frac h2 A_2\Big).
\label{eq:PR}
\end{equation} 
\begin{theorem} \label{thm:PR_DR_orders}
Consider the approximate solution $S_h^m\eta$ given by applying either the Douglas--Rachford scheme \eqref{eq:DR} or the Peaceman--Rachford scheme \eqref{eq:PR} to the linear evolution equation \eqref{eq:evolution_equation}. If Assumptions \ref{as:dissipativity} and \ref{as:splitting} are valid, $h\max\{M[A_1],\, M[A_2]\} \leq 1/2$ and the solution $u$ is regular enough with $u(t) \in \DD(A)$ for all times $t \in [0,T]$, then the global error can be bounded as
\[ \norm{u(mh) - S_h^m\eta} \leq 5h^p\e^{3T(M[A_1]+M[A_2])} \sum_{i=0}^p \big\lVert A_1^{p-i}u^{(i+1)}\big\rVert_{L^1(0,T;\HH)}, \quad mh \leq T, \]
with first-order convergence, $p = 1$, for both schemes or second-order convergence, $p = 2$, for the Peaceman--Rachford scheme. Here $u^{(i+1)}$ denotes derivative number $i+1$ of $u$.
\end{theorem}
We do not prove this theorem here but refer to the proof of \cite[Theorem~2]{Hansen2013} for the Peacman--Rachford scheme and the proof of \cite[Lemma~2]{Hansen2015} for the Douglas--Rachford scheme. The latter proof needs to be extended to the infinite dimensional setting which is straight-forward to do using the same techniques as is used in the former proof. Further, in \cite{Hansen2013} it is assumed that $\DD(A) = \DD(A_1) \cap \DD(A_2)$, however it is easy to see that the same proof also holds under the weaker assumption that $\DD(A) \subset \DD(A_1) \cap \DD(A_2)$.

%

As noted in the introduction the Douglas--Rachford and the Peaceman--Rachford schemes have favorable local error structures. A consequence is that, for linear two-operator splittings and under Assumptions \ref{as:dissipativity} and \ref{as:splitting}, regular initial data $\eta \in \DD(A^2)$ is enough for first-order convergence. The proof is based on the semigroup theory which we will not pursue in this paper, cf.\ \cite{Pazy}. For less regular initial data $\eta \in \DD(A_1)$ or $\eta \in \DD(A_2)$ we still have convergence, but possibly without order. However, in this case there may be no classical solution meaning that the latter convergence should be understood in the context of mild solutions, see \eqref{eq:mild_solution}. For a proof see \cite[Theorem~3]{Hansen2013}.

Thus, for domain decomposition operator splittings of the PDE \eqref{eq:rda} we get first-order convergence for the ADI schemes \eqref{eq:DR} and \eqref{eq:PR} without any other regularity assumptions on $u$ than that the initial data $\eta$ is in $\DD(A^2)$. This follows as a direct consequence of the discussion in the preceding paragraph and Lemmas \ref{lem:max_diss} and \ref{lem:splitting}.

\subsection{A first-order additive splitting scheme} \label{sec:sum_splitting}
Contrary to the ADI methods the additive splitting scheme \eqref{eq:sum_splitting} is stable for splittings of the evolution equation \eqref{eq:evolution_equation} with any number $q$ of linear maximal dissipative operators. We note that for domain decompositions it is most common to use more than two domains (resulting in more than two operators) as it gives more flexibility when constructing the decomposition, cf.\ Figures \ref{fig:domain_decomposition}a and \ref{fig:domain_decomposition}c. 
Let $M = \max\{M[A_1],\dots, M[A_q]\}$, we will prove the following.


\begin{theorem} \label{thm:sum_split_orders}
Consider the approximate solution $S_h^m\eta$ given by applying the additive splitting scheme \eqref{eq:sum_splitting} to the linear evolution equation \eqref{eq:evolution_equation}. Further, assume that Assumptions \ref{as:dissipativity} and \ref{as:splitting} are valid, that $hqM \leq 1/2$ and that the solution $u$ is regular enough with $u(t) \in \DD(A)$ for all times $t \in [0,T]$. Then the additive splitting scheme \eqref{eq:sum_splitting} is first-order convergent and the global error can be bounded as
\[ \norm{u(mh) - S_h^m\eta} \leq h \e^{2qMT} \Big(\norm{\ddot u}_{L^{1}(0,T;\HH)} + 2qT \sum_{k=1}^q \norm{A_k^2u}_{\CC([0,T];\HH)} \Big), \quad mh \leq T.\]
\end{theorem}
\begin{proof}
For $k = 1,\dots, q$ we introduce the notation
\[ a_k = hqA_k \quad \text{and} \quad \alpha_k = (I-a_k)\inv. \]
As Assumption~\ref{as:dissipativity} is valid and $hqM \leq 1/2$, the identities
\begin{equation}
I = \alpha_k - a_k \alpha_k
\label{eq:identity_expansion}
\end{equation}
hold on all of $\HH$. Furthermore, note that $a_k\alpha_k = \alpha_k a_k$ on $\DD(A_k)$. In the new notation the time stepping operator of the additive splitting scheme \eqref{eq:sum_splitting} can be written as
\[ S_h = \frac1q\sum_{k=1}^q{\alpha_k}. \]

Let $t_j = jh$ and expand the global error by the telescopic sum
\begin{equation}
\norm{u(mh) - S_h^m} = \big\lVert \sum_{j=1}^{m}{S_h^{m-j}\uj - S_h^{m-j+1}\ujm} \big\rVert \leq \sum_{j=1}^{m}{L[S_h^{m-j}] \norm{\uj - S_h\ujm}}.
\label{eq:proof_sum_split_telescopic} \end{equation}
Stability is a direct consequence of the resolvent bound \eqref{eq:resolvent_bound} as
\begin{equation} 
L[S_h] = L\big[\frac{1}{q}\sum_{k=1}^q \alpha_k\big] \leq \frac{1}{q}\sum_{k=1}^q\frac{1}{1-hqM[A_k]} \leq \frac{1}{1-hqM},
\label{eq:proof_sum_split_one_step_stability}
\end{equation} 
which gives
\begin{equation}
L[S_h^{m-j}] \leq L\big[\frac{1}{q}\sum_{k=1}^q \alpha_k\big]^{m-j} \leq \left(\frac{1}{1-hqM}\right)^{m-j} \leq \e^{2qMT},
\label{eq:proof_sum_split_stability}
\end{equation}
where we have used the inequality $1/(1-x) \leq \e^{2x}$ for all $x \in [0,1/2]$.

Since $\ujm \in \DD(A) \subset \bigcap_{k=1}^q \DD(A_k)$ we can use the commutativity $a_k\alpha_k = \alpha_k a_k$ and twice the identity $\alpha_k = I + a_k \alpha_k$ to expand the local error as
\begin{equation} 
\begin{aligned}
\uj - S_h\ujm &=  \uj - \frac1q \sum_{k=1}^q{(I + a_k + a_k \alpha_k a_k) \ujm} \\
	&= \uj - \ujm - hAu(t_{j-1}) - \frac1q \sum_{k=1}^q{a_k \alpha_k a_k\ujm}.
\end{aligned} \label{eq:proof_error_expansion}
\end{equation}
Then, for a sufficiently regular $u$ we can rewrite the error as a sum of a quadrature error $q_j$ and a splitting error $s_j$:
\begin{align}
&q_j = \uj - \ujm - hAu(t_{j-1}) = \uj - \ujm - h\dot u(t_{j-1}) = h\int_{t_{j-1}}^{t_j} \frac{t_{j}-\tau}{h} \ddot u(\tau) \dtau, \label{eq:proof_q_j} \\
&s_j = - \frac1q \sum_{k=1}^q{a_k \alpha_k a_k\ujm} = - h^2q\sum_{k=1}^q{\alpha_kA_k^2\ujm}. \label{eq:proof_s_j}
\end{align}
Finally, since the constants $L[\alpha_k]$ are bounded by 2 when $hqM \leq 1/2$, the local error can be bounded as
\[ \norm{\uj - S_h\ujm} \leq h \Big(\int_{t_{j-1}}^{t_j} \norm{\ddot u(\tau)} \dtau + 2hq\sum_{k=1}^q \norm{A_k^2 \ujm} \Big). \]
Combining with the stability \eqref{eq:proof_sum_split_stability} in the telescopic sum \eqref{eq:proof_sum_split_telescopic} the sought after global error bound is achieved.
\end{proof}

\begin{remark}
In \cite{Hansen2015} ADI splitting schemes are combined with convergent spatial discretizations like, e.g., classical piecewise linear finite elements. Simultaneous space-time convergence orders are proven when a linear maximal dissipative operator is split into two operators. We note that for splittings of the linear evolution equation \eqref{eq:evolution_equation} with arbitrary number of operators $q$ the same space-time convergence result holds also for the additive splitting scheme \eqref{eq:sum_splitting} with first-order convergence in time. For details we refer to the assumptions, theorems and lemmas in the aforementioned article. Only Assumption~2 and Theorem~3 of that article require modifications, and these are all minor.
\end{remark}

\subsection{First-order additive splitting schemes for semilinear evolution equations} \label{sec:sum_splitting_semilinear}
The first-order additive splitting scheme \eqref{eq:sum_splitting} can easily be extended to treat semilinear evolution equations of the type
\begin{equation}
\dot u = (A + F)u = \Big(\Big(\sum_{k=1}^q A_k\Big) + F\Big)u, \quad u(0) = \eta,
\label{eq:evolution_equation_semilinear}
\end{equation}
where $F$ is a nonlinear operator. 
In the abstract setting of maximal dissipative operators we will prove a semilinear extension of Theorem~\ref{thm:sum_split_orders} for the two schemes
\begin{subequations}
\begin{align}
&S_h = (I-hF)\inv\Big(\frac1q \sum_{k=1}^q (I-hqA_k)\inv\Big) \label{eq:sum_splitting_semilinear_dis} \quad \text{and} \\
&S_h = \Big(\frac1q \sum_{k=1}^q (I-hqA_k)\inv\Big)(I+hF) \label{eq:sum_splitting_semilinear_Lip}
\end{align}
\end{subequations}
under the corresponding assumptions:
\begin{assumption} \label{as:semilinear}
The operator $F:\DD(F) \subset \HH \rightarrow \HH$ satisfies one of the following two statements:
\begin{itemize}[noitemsep,topsep=0pt]
	\item[(a)] $F$ is maximal dissipative on $\HH$; or
	\item[(b)] $F$ is Lipschitz continuous on $\HH$.
\end{itemize}
Furthermore, the operator $A+F: \DD(A) \cap \DD(F) \subset \HH \rightarrow \HH$ is maximal dissipative on $\HH$.
\end{assumption}
That is, the scheme \eqref{eq:sum_splitting_semilinear_dis} is suited for stiff nonlinear vector fields $F$ whereas for nonstiff $F$ the scheme \eqref{eq:sum_splitting_semilinear_Lip} can be used instead for simpler computations.
\begin{example_} \label{ex:semilinear_parabolic}
Consider the parabolic model problem of Section~\ref{sec:domain_decomposition} with $\Omega \subset \R^d$, $d = 1,2,3$, and let $\rho = \sigma = 0$. Further, let $F: L^{2p}(\Omega) \subset L^2(\Omega) \rightarrow L^2(\Omega)$ be a nonlinear potential defined by $Fv = v - v^p$ for an odd integer $p \geq 3$. Then $F$ and $A+F: \DD(A) \cap \DD(F) \subset L^2(\Omega) \rightarrow L^2(\Omega)$ are maximal dissipative on $L^2(\Omega)$. Therefore, when splitting $A$ by a domain decomposition as in Section~\ref{sec:domain_decomposition}, Assumptions \ref{as:dissipativity}, \ref{as:splitting} and \ref{as:semilinear}a are all valid. Thus, as the action of the resolvent of $F$ is straight-forward to parallelize, efficient parallel computations can be carried out using the additive splitting scheme \eqref{eq:sum_splitting_semilinear_dis}. For details see \cite[Theorem~5.5]{Barbu} and \cite[Section~6]{Hansen2013}.
\end{example_}
\begin{example_} \label{ex:Caginalp}
We exemplify also with a semilinear application with a linear term which is not on the form \eqref{eq:rda}. That is, consider on $\Omega \subset \R^d$, $d = 1,2,3$, the phase field model
\[ \dot u = Au + Fu = \bp \Delta & -c\Delta \\ 0 & \Delta \ep \bp u_1 \\ u_2 \ep + \bp 0 \\ (1-c)u_2 - u_2^3 + u_1 \ep, \]
which is commonly encountered when modeling solidification processes. Here $c$ is a positive constant, $u$ is a vector-valued function $u = (u_1, u_2)\trans$ and the boundary of $\Omega$ fulfills the regularity properties assumed in Section~\ref{sec:domain_decomposition}. Equip the equation with a suitable initial condition and either homogeneous Dirichlet or Neumann conditions and define a domain decomposition operator splitting by letting $A_k$ be $A$ but with all occurrences of $\Delta$ replaced by $\nabla\cdot\chi_k\nabla$. Then, a similar analysis as in Section~\ref{sec:analytic_setting} can be performed to show that Assumptions \ref{as:dissipativity}, \ref{as:splitting} and \ref{as:semilinear}a are valid with $\HH = (L^2(\Omega))^2$. This defines a splitting fit for efficient parallel computations. See also the references in the preceding example.
\end{example_}

Since $A+F$ is nonlinear the existence of a classical solution to the evolution equation \eqref{eq:evolution_equation_semilinear} can not be guaranteed. However, under Assumption~\ref{as:semilinear} there is a unique mild solution for every $\eta$ in the closure of $\DD(A)\cap\DD(F)$. This mild solution is defined by
\begin{equation}
u(t) = \lim_{m\rightarrow \infty} \left(I - \frac tm (A+F)\right)^{-m} \eta,
\label{eq:mild_solution}
\end{equation}
for $t > 0$. Nonlinear maximal dissipative operators and associated evolution equations are surveyed in \cite[Sections 3.1 and 4.1]{Barbu}.

\begin{theorem} \label{thm:sum_split_semilinear_orders}
Consider the approximate solution $S_h^m\eta$ given by applying one of the additive splitting schemes \eqref{eq:sum_splitting_semilinear_dis} or \eqref{eq:sum_splitting_semilinear_Lip} to the semilinear evolution equation \eqref{eq:evolution_equation_semilinear}. Further, assume that Assumptions \ref{as:dissipativity} and \ref{as:splitting} are valid, that $hqM \leq 1/2$ and that the solution $u$ is regular enough with $u(t) \in \DD(A)$ for all times $t \in [0,T]$. Then, if Assumption~\ref{as:semilinear}a is valid and $hM[F] \leq 1/2$, the scheme \eqref{eq:sum_splitting_semilinear_dis} is first-order convergent with global error bounded as
\[ \norm{u(mh) - S_h^m\eta} \leq 2h \e^{2T(qM + M[F])} \Big(\sum_{i=0}^1 \big\lVert A^{1-i}u^{(i+1)}\big\rVert_{L^1(0,T;\HH)} + 2qT \sum_{k=1}^q \norm{A_k^2u}_{\CC([0,T];\HH)} \Big), \]
and, if instead Assumption~\ref{as:semilinear}b is valid, the scheme \eqref{eq:sum_splitting_semilinear_Lip} is first-order convergent with global error bounded as
\[ \norm{u(mh) - S_h^m\eta} \leq h \e^{T(2qM + L[F])} \Big(\norm{\ddot u}_{L^1(0,T;\HH)} + 2T\sum_{k=1}^q \norm{A_kFu + qA_k^2u}_{\CC([0,T];\HH)}\Big), \]
where $mh \leq T$.
\end{theorem}
\begin{proof}
We introduce the notation
\[ f = hF \quad \text{and} \quad \vphi = (I-f)\inv, \]
and note that the identity \eqref{eq:identity_expansion} applies also with $a_k$ and $\alpha_k$ replaced with $f$ and $\vphi$. In this notation the schemes \eqref{eq:sum_splitting_semilinear_dis} and \eqref{eq:sum_splitting_semilinear_Lip} can be written as
\[ S_h = \vphi\Big(\sumsplit\Big), \quad \text{respectively} \quad S_h = \Big(\sumsplit\Big)(I+f). \]

We discuss the necessary modifications to the proof of Theorem~\ref{thm:sum_split_orders}. As the resolvent bound \eqref{eq:resolvent_bound} holds for maximal dissipative $F$, the stability of the scheme \eqref{eq:sum_splitting_semilinear_dis} under Assumption~\ref{as:semilinear}a follows as in the linear case but with $L[S_h^{m-j}] \leq\e^{2T(qM + M[F])}$. For the consistency we bound the local error as
\[ \norm{\uj - S_h\ujm} \leq L[\vphi] \norm{(I-f)\uj - (I-f)S_h\ujm}, \]
where $L[\vphi] \leq 2$. Expanding $(I-f)\uj - (I-f)S_h\ujm$ as in \eqref{eq:proof_error_expansion} and then subtracting $hA(\uj-\ujm)$ from the quadrature error and adding it to the splitting error gives
\[ \begin{aligned}
&q_j = \uj - \ujm - h (F\uj + A\uj) = -h\int_{t_{j-1}}^{t_j} \frac{\tau-t_{j-1}}{h} \ddot u(\tau) \dtau,\\
&s_j = hA(\uj-\ujm) - h^2q\sum_{k=1}^q{\alpha_k A_k^2\ujm},
\end{aligned} \]
which holds for a sufficiently regular $u$. Further, note that 
\[ A(\uj-\ujm) = A \int_{t_{j-1}}^{t_j} \dot u(\tau) \dtau = \int_{t_{j-1}}^{t_j} A \dot u(\tau) \dtau, \]
where the last equality holds since the operator $A$ is closed as a consequence of being linear and maximal dissipative. 

Now, for Lipschitz continuous $F$ we have $L[I+f] \leq 1 + hL[F]$ implying $L[I+f]^{m-j} \leq \e^{TL[F]}$ which gives the stability for the scheme \eqref{eq:sum_splitting_semilinear_Lip}. Using the regularity of $u$ and expanding the local error as in \eqref{eq:proof_error_expansion} gives
\[ \uj - S_h\ujm = \uj - \frac1q\sum_{k=1}^q{(I + a_k + a_k \alpha_k a_k)\ujm} - \frac1q\sum_{k=1}^q{\left(I + a_k\alpha_k \right)f\ujm}. \]
Thus, we get the same quadrature error $q_j$ as in \eqref{eq:proof_q_j} and the splitting error $s_j$ of \eqref{eq:proof_s_j} is modified by adding the term $-h^2\sum_{k=1}^q\alpha_kA_kF \ujm$. The rest follows as in the proof of Theorem~\ref{thm:sum_split_orders}.
\end{proof}

\begin{remark}
The concept of dissipativity can be generalized to spaces $\HH$ which are merely Banach spaces, see e.g.\ \cite[Section~3.1]{Barbu}. With this generalization the resolvent of a maximal dissipative operator still exists and fulfills the bound \eqref{eq:resolvent_bound}. Thus, inequality \eqref{eq:proof_sum_split_one_step_stability} still holds, the additive splitting schemes \eqref{eq:sum_splitting_semilinear_dis} and \eqref{eq:sum_splitting_semilinear_Lip} are still stable, and the global error bounds of Theorems \ref{thm:sum_split_orders} and \ref{thm:sum_split_semilinear_orders} still hold. However, Theorem~\ref{thm:PR_DR_orders} can not be extended to this Banach setting since the stability of the Douglas--Rachford and Peaceman--Rachford schemes require more structure, cf.\ \cite[page~1906]{Hansen2013}.
\end{remark}

Even when the regularity required by Theorems \ref{thm:sum_split_orders} or \ref{thm:sum_split_semilinear_orders} is not present we may still have convergence without order to the mild solution \eqref{eq:mild_solution}. 
\begin{theorem} \label{thm:sum_split_semilinear_wo_orders}
Consider the approximate solution $S_h^m\eta$ given by applying one of the additive splitting schemes \eqref{eq:sum_splitting_semilinear_dis} or \eqref{eq:sum_splitting_semilinear_Lip} to the semilinear evolution equation \eqref{eq:evolution_equation_semilinear}. Further, assume that Assumptions \ref{as:dissipativity} and \ref{as:splitting} are valid and that $\DD(A)\cap\DD(F)$ is dense in $\HH$. Then, under Assumption~\ref{as:semilinear}a the scheme \eqref{eq:sum_splitting_semilinear_dis} is convergent, i.e.\
\[ \lim_{m\rightarrow\infty} S^m_{t/m} \eta = u(t), \]
for every $\eta \in \HH$ and $t \geq 0$. Furthermore, if $\DD(F) = \HH$ the same holds for the scheme \eqref{eq:sum_splitting_semilinear_Lip} under Assumption~\ref{as:semilinear}b.
\end{theorem}

\begin{proof}
First, let $F$ be maximal dissipative and note that it is densely defined since $\DD(A)\cap\DD(F) \subset \DD(F)$ is dense in $\HH$. 
In this setting \cite[Lemma~3]{Lions} gives, for each $v \in \DD(F)$, that
\[\lim_{h\rightarrow 0} F\vphi v_h = Fv\]
for each family $\{v_h\}_h \subset \HH$ such that there is a $z \in \HH$ for which $(v_h - v)/h$ tends to $z$ when $h$ tends to zero. Of course the same type of limit also holds for the operators $A_1, \dots, A_q$. Now, for $v \in \DD(A) \subset \bigcap_{k=1}^q \DD(A_k)$ let $v_h = 1/q\cdot\sum_{k=1}^q \alpha_kv$, we have
\[ \frac1h (v_h - v) = \frac1h \Big( \frac1q\sum_{k=1}^q\left(I+a_k\alpha_k\right)v - v \Big) = \sum_{k=1}^q A_k \alpha_k v \rightarrow \sum_{k=1}^q A_k v = Av, \quad \text{as } h\rightarrow 0. \]
Then, the scheme \eqref{eq:sum_splitting_semilinear_dis} is consistent, i.e.,
\[ \frac1h (S_h-I)v = \frac1h (\vphi - I) v_h + \frac1h (v_h - v) = F\vphi v_h + \sum_{k=1}^q A_k \alpha_k v \rightarrow (F+A)v, \quad \text{as } h\rightarrow 0, \]
for every $v \in \DD(A)\cap\DD(F)$. Therefore, since $S_h$ is stable in the sense $L[S_h] \leq 1 + \Ordo(h)$ the scheme \eqref{eq:sum_splitting_semilinear_dis} is convergent according to \cite[Corollary~4.3]{Brezis}. 

Now, instead let Assumption~\ref{as:semilinear}b be valid and let $\DD(F) = \HH$. Let $S_h$ be given by the scheme \eqref{eq:sum_splitting_semilinear_Lip}, then $S_h$ is defined on all of $\HH$. Furthermore, the scheme is consistent as
\[ \frac1h (S_h-I)v = \frac1h \Big( \frac1q\sum_{k=1}^q\left(I+a_k\alpha_k + \alpha_kf\right)v - v \Big) = \sum_{k=1}^q \Big(A_k \alpha_k + \frac1q\alpha_k F \Big) v  \rightarrow (A+F)v,  \quad \text{as } h\rightarrow 0, \]
for every $v \in \DD(A)$. Here we have used that $\alpha_k w$ tends to $w$ for $w$ in $\HH$ and maximal dissipative $A_k$, cf.\ \cite[Proposition~11.3]{Deimling1985}. Then, combining with the stability given by inequality \eqref{eq:proof_sum_split_one_step_stability} and the Lipschitz continuity of $F$ we again get the sought after convergence result from \cite[Corollary~4.3]{Brezis}.
\end{proof}

Note that $\DD(A)$ is dense in $\HH$ as a direct consequence of Assumption~\ref{as:dissipativity}. Thus, in the linear ($F=0$) maximal dissipative setting, the additive splitting scheme \eqref{eq:sum_splitting} always converges. That is, for domain decomposition operator splittings of the PDE \eqref{eq:rda} this scheme is always convergent as a direct consequence of Lemmas \ref{lem:max_diss} and \ref{lem:splitting}, without any further assumptions.

\subsection{A second-order additive fractional step Crank--Nicolson scheme} \label{sec:fractional_step_CN}
As a second-order alternative to the additive splitting scheme \eqref{eq:sum_splitting} we consider a fractional step Crank--Nicolson method. Just as for the first-order scheme \eqref{eq:sum_splitting} any number of linear operators are allowed. Furthermore, there is some additive structure that can be used for implementation on parallel computing systems. For brevity, we will define the second-order scheme in the following notation:
\[ a_k = \frac{h}{2}A_k \quad \text{and} \quad \alpha_k = (I-a_k)\inv, \quad \text{for } k = 1,\dots,q.\]
Note that, under Assumption~\ref{as:dissipativity} and with $hM[A_k] < 2$, the identity \eqref{eq:identity_expansion} holds, which means that the operators $a_k\alpha_k$ are bounded. Furthermore, since the domains $\DD(A_k)$ are dense in $\HH$ (as per Assumption~\ref{as:dissipativity}), the operators $\alpha_k a_k$ can be extended to bounded operators on all of $\HH$ using the relationship $a_k\alpha_k = \alpha_k a_k$ which holds on $\DD(A_k)$ (and after the extension also on $\HH$). Thus, the fractional step Crank--Nicolson scheme,
\begin{equation}
S_h\; =\; \frac12 \alpha_q(I+a_q) \alpha_{q-1}(I+a_{q-1}) \cdots \alpha_1(I+a_1)\; +\; \frac12 \alpha_1(I+a_1) \alpha_{2}(I+a_{2}) \cdots \alpha_q(I+a_q),
\label{eq:fractional_step_CN}
\end{equation}
is well-defined. The second-order convergence is given as follows.
\begin{theorem} \label{thm:fractional_step_CN_orders}
Consider the approximate solution $S_h^m\eta$ given by applying the scheme \eqref{eq:fractional_step_CN} to the linear evolution equation \eqref{eq:evolution_equation}. Further, assume that Assumptions \ref{as:dissipativity} and \ref{as:splitting} are valid, that $hM \leq 1$ and that the solution $u$ is regular enough with $u(t) \in \DD(A)$ for all times $t \in [0,T]$. Then the fractional step Crank--Nicolson scheme \eqref{eq:fractional_step_CN} is second-order convergent and the global error can be bounded as
\[ \begin{aligned}
\norm{u(mh) - S_h^m\eta} &\leq Ch^2 \Big(\lVert u^{(3)}\rVert_{L^1(0,T;\HH)} + \sum_{k=1}^q \sum_{\ell=1}^q \sum_{i=1}^2 \norm{A_k^i A_{\ell}^{3-i} u}_{\CC([0,T];\HH)} \\
	&+ \sum_{k=1}^q\; \sum_{\ell=k+1}^q\; \sum_{r=\ell+1}^q \big( \norm{A_k A_{\ell} A_{r} u}_{\CC([0,T];\HH)} + \norm{A_{r} A_{\ell} A_{k} u}_{\CC([0,T];\HH)} \big) \Big), 
\end{aligned} \]
where $mh \leq T$ and $C$ is a constant that only depends on $q$, $M[A_1], \dots, M[A_q]$ and $T$.
\end{theorem}

\begin{proof}
Consider the telescopic global error expansion \eqref{eq:proof_sum_split_telescopic}.
According to the proof of \cite[Lemma~1]{Hansen2013}, when $hM[A_k] \leq 1$, we have
\[ L[\alpha_k (I + a_k)] \leq \e^{3/2\, hM[A_k]}, \quad \text{for } k = 1, \dots, q. \]
Then stability follows as
\[ L[S_h^{m-j}] \leq \e^{3/2\, (m-j)h \sum_{k=1}^q M[A_k]} \leq \e^{3/2\, T \sum_{k=1}^q M[A_k]}. \]

For the consistency proof we need the following identity, which holds for a sufficiently regular $u$:
\begin{equation}
\alpha_k a_k \alpha_{\ell} a_{\ell} \ujm = \alpha_k a_k (a_{\ell} + a_{\ell}^2 + \alpha_{\ell} a_{\ell}^3) \ujm \!=\! (a_k a_{\ell} + \alpha_k a_k^2 a_{\ell} + \alpha_k a_k a_{\ell}^2 + \alpha_k a_k \alpha_{\ell} a_{\ell}^3) \ujm,
\label{eq:proof_FSCN_identity}
\end{equation}
for $k,\ell = 1,\dots, q$. Note how identity \eqref{eq:identity_expansion} has been applied multiple times and how using it again lets us rewrite the local error as
\begin{equation} \begin{aligned}
&\uj - S_h\ujm \\
	&\quad= \uj - \Big(\frac12 (I + 2\alpha_qa_q) \cdots (I + 2\alpha_1a_1) + \frac12 (I + 2\alpha_1a_1) \cdots (I + 2\alpha_qa_q)\Big)\ujm. 
\end{aligned} \label{eq:proof_FSCN_reformulated_method}
\end{equation}
We expand \eqref{eq:proof_FSCN_reformulated_method} and consider first the terms composed of at most two operators $\alpha_{(\cdot)} a_{(\cdot)}$, i.e.\ in the compositions above consider first the terms given by choosing the identity all but zero, one or two times when expanding. This gives
\[ \begin{aligned}
&\Big( I + 2\sum_{k=1}^q \alpha_k a_k + 2\sum_{k=1}^q\sum_{\substack{\ell=1\\\ell\neq k}}^q \alpha_k a_k \alpha_{\ell} a_{\ell} \Big) \ujm \\
	&\quad=\Big( I + 2\sum_{k=1}^q a_k + a_k^2 + \alpha_k a_k^3 + 2\sum_{k=1}^q\sum_{\substack{\ell=1\\\ell\neq k}}^q a_k a_{\ell} + \alpha_k a_k^2 a_{\ell} + \alpha_k a_k a_{\ell}^2 + \alpha_k a_k \alpha_{\ell} a_{\ell}^3 \Big) \ujm.
\end{aligned} \]
Using this expansion, the local error can be written as a sum of a quadrature error and a splitting error
\[ \begin{aligned}
&q_j = \uj - \ujm - h \sum_{k=1}^q A_k \ujm - \frac{h^2}{2} \Big(\sum_{k=1}^q A_k\Big)^2 \ujm = \frac{h^2}{2} \int_{t_{j-1}}^{t_j} \Big(\frac{t_j-\tau}{h}\Big)^2 u^{(3)}(\tau) \dtau, \\
&s_j = \frac{h^3}{4} \Big( \sum_{k=1}^q \alpha_k A_k^3 + \sum_{k=1}^q\sum_{\substack{\ell=1\\\ell\neq k}}^q \alpha_k A_k^2A_\ell + \alpha_k A_kA_\ell^2 + \alpha_ka_k\alpha_\ell A_\ell^3) \Big) \ujm + r_j,
\end{aligned} \]
where $r_j$ consists of the terms composed of three or more operators $\alpha_{(\cdot)} a_{(\cdot)}$. 
That is, the terms given by choosing $I$ at most $q-3$ times in the expansion of \eqref{eq:proof_FSCN_reformulated_method}. Each of these terms, 
can be written either as
\begin{equation}
2^{K-1} \alpha_{k_K} a_{k_K} \cdots \alpha_{k_1} a_{k_1} \ujm \quad \text{or as} \quad 2^{K-1} \alpha_{k_1} a_{k_1} \cdots \alpha_{k_K} a_{k_K} \ujm,
\label{eq:proof_FSCN_HOT}
\end{equation}
for some integers $1 \leq k_1 < k_2 < \dots < k_K \leq q$ and some integer $K = 3, \dots, q$. Then, for a sufficiently regular $u$, the identity \eqref{eq:proof_FSCN_identity} gives that any term on the first form in \eqref{eq:proof_FSCN_HOT} can be written as
\[ \begin{aligned} 
&2^{K-1} \alpha_{k_K} a_{k_K} \cdots \alpha_{k_1} a_{k_1} \ujm = \\
&\quad h^3 2^{K-4} \alpha_{k_K} a_{k_K} \cdots \alpha_{k_4} a_{k_4} \big(\alpha_{k_3} A_{k_3} A_{k_2} A_{k_1} + \alpha_{k_3} a_{k_3} (\alpha_{k_2} A_{k_2}^2 A_{k_1} + \alpha_{k_2} A_{k_2} A_{k_1}^2 + \alpha_{k_2} a_{k_2} \alpha_{k_1} A_{k_1}^3 ) \big) \ujm. 
\end{aligned} \]
Of course similar expansions hold for the terms on the second form in \eqref{eq:proof_FSCN_HOT}. Finally, we apply the norm to the local error and the sought after global error bound follows since, for all $k = 1,\dots, q$, we have $L[\alpha_k] \leq 2$ and $L[\alpha_k a_k] \leq L[I] + L[\alpha_k] \leq 3$ when $hM[A_k] \leq 1$.
\end{proof}

%

\section*{Acknowledgments}
The authors thank Philipp Birken for helpful input on domain decomposition methods during the initial preparation of the paper. 


\bibliographystyle{IMANUM-BIB}
\bibliography{references}

\begin{thebibliography}{}

\bibitem[Adams \& Fournier(2003)Adams \& Fournier]{Adams}
{\sc Adams, R. \& Fournier, J.} (2003)
\newblock {\em Sobolev Spaces\/}. Pure and Applied Mathematics,  vol. 140.
\newblock Amsterdam: Academic Press.

\bibitem[Arrarás \& Portero(2015)Arrarás \& Portero]{Arraras2015}
{\sc Arrarás, A. \& Portero, L.} (2015)
\newblock Improved accuracy for time-splitting methods for the numerical
  solution of parabolic equations.
\newblock {\em Appl. Math. Comput.}, --.

\bibitem[Barbu(2010)Barbu]{Barbu}
{\sc Barbu, V.} (2010)
\newblock {\em Nonlinear Differential Equations of Monotone Types in Banach
  Spaces\/}.
\newblock Springer Monographs in Mathematics.
\newblock New York: Springer.

\bibitem[Brenner \& Scott(2008)Brenner \& Scott]{Brenner}
{\sc Brenner, S. \& Scott, R.} (2008)
\newblock {\em The Mathematical Theory of Finite Element Methods\/}. Texts in
  Applied Mathematics,  vol.~15.
\newblock New York: Springer.

\bibitem[Brezis \& Pazy(1972)Brezis \& Pazy]{Brezis}
{\sc Brezis, H. \& Pazy, A.} (1972)
\newblock Convergence and approximation of semigroups of nonlinear operators in
  {B}anach spaces.
\newblock {\em J. Funct. Anal.}, {\bf 9}, 63--74.

\bibitem[Coron(1982)Coron]{Coron1982}
{\sc Coron, J.-M.} (1982)
\newblock Formules de trotter pour une équation d'évolution quasilinéaire du
  1er ordre.
\newblock {\em J. Math. Pures Appl.}, {\bf 61}, 91--112.

\bibitem[Deimling(1985)Deimling]{Deimling1985}
{\sc Deimling, K.} (1985)
\newblock {\em Nonlinear Functional Analysis\/}.
\newblock Berlin; New York: Springer.

\bibitem[Douglas(1955)Douglas]{Douglas1955}
{\sc Douglas, J.} (1955)
\newblock On the numerical integration of $\frac{\partial^2 u} {\partial x^2} +
  \frac{\partial^2 u}{\partial y^2} = \frac{\partial u} {\partial t}$ by
  implicit methods.
\newblock {\em J. Soc. Indust. Appl. Math.}, {\bf 3}, 42--65.

\bibitem[Douglas \& Gunn(1964)Douglas \& Gunn]{Douglas1964}
{\sc Douglas, Jim, J. \& Gunn, J.~E.} (1964)
\newblock A general formulation of alternating direction methods.
\newblock {\em Numer. Math.}, {\bf 6}, 428--453.

\bibitem[Engel \& Nagel(2000)Engel \& Nagel]{Engel}
{\sc Engel, K. \& Nagel, R.} (2000)
\newblock {\em One-Parameter Semigroups for Linear Evolution Equations\/}.
  Graduate Texts in Mathematics,  vol. 194.
\newblock New-York: Springer.

\bibitem[Evans(2010)Evans]{Evans2010}
{\sc Evans, L.} (2010)
\newblock {\em Partial Differential Equations\/}.
\newblock Graduate Studies in Mathematics.
\newblock Providence: American Mathematical Society.

\bibitem[Gordeziani \& Meladze(1974)Gordeziani \& Meladze]{Gordeziani1974}
{\sc Gordeziani, D. \& Meladze, G.} (1974)
\newblock Simulation of the third boundary value problem for multidimensional
  parabolic equations in an arbitrary domain by one-dimensional equations.
\newblock {\em {USSR} Comp. Math. Math.+\/}, {\bf 14}, 249--253.

\bibitem[Grisvard(1985)Grisvard]{Grisvard}
{\sc Grisvard, P.} (1985)
\newblock {\em Elliptic Problems in Nonsmooth Domains\/}. Monographs and
  Studies in Mathematics,  vol.~24.
\newblock London: Pitman.

\bibitem[Hackbusch(1992)Hackbusch]{Hackbusch}
{\sc Hackbusch, W.} (1992)
\newblock {\em Elliptic Differential Equations: Theory and Numerical
  Treatment\/}. Springer Series in Computational Mathematics,  vol.~18.
\newblock Berlin: Springer.

\bibitem[Hairer {\em et~al.}(2006)Hairer, Lubich, \& Wanner]{Hairer}
{\sc Hairer, E., Lubich, C. \& Wanner, G.} (2006)
\newblock {\em Geometric Numerical Integration: Structure-Preserving Algorithms
  for Ordinary Differential Equations\/}. Springer Series in Computational
  Mathematics,  vol.~31.
\newblock Berlin: Springer.

\bibitem[Hansen \& Henningsson(2013)Hansen \& Henningsson]{Hansen2013}
{\sc Hansen, E. \& Henningsson, E.} (2013)
\newblock A convergence analysis of the {P}eaceman--{R}achford scheme for
  semilinear evolution equations.
\newblock {\em SIAM J. Numer. Anal.}, {\bf 51}, 1900--1910.

\bibitem[Hansen \& Henningsson(2015)Hansen \& Henningsson]{Hansen2015}
{\sc Hansen, E. \& Henningsson, E.} (2015)
\newblock A full space-time convergence order analysis of operator splittings
  for linear dissipative evolution equations.
\newblock To appear in Communications in Computational Physics (SCPDE14 special
  issue).

\bibitem[Hundsdorfer \& Verwer(1989)Hundsdorfer \& Verwer]{Hundsdorfer1989}
{\sc Hundsdorfer, W. \& Verwer, J.} (1989)
\newblock Stability and convergence of the {P}eaceman--{R}achford {ADI} method
  for initial-boundary value problems.
\newblock {\em Math. Comp.}, {\bf 53}, 81--101.

\bibitem[Hundsdorfer \& Verwer(2003)Hundsdorfer \& Verwer]{Hundsdorfer}
{\sc Hundsdorfer, W. \& Verwer, J.} (2003)
\newblock {\em Numerical Solution of Time-Dependent
  Advection-Diffusion-Reaction Equations\/}. Springer Series in Computational
  Mathematics,  vol.~33.
\newblock New York: Springer.

\bibitem[Kufner(1980)Kufner]{Kufner1980}
{\sc Kufner, A.} (1980)
\newblock {\em Weighted {S}obolev Spaces\/}. Teubner-Texte zur Mathematik,
  vol.~31.
\newblock Leipzig: Teubner.

\bibitem[Lions \& Mercier(1979)Lions \& Mercier]{Lions}
{\sc Lions, P.~L. \& Mercier, B.} (1979)
\newblock Splitting algorithms for the sum of two nonlinear operators.
\newblock {\em SIAM J. Numer. Anal.}, {\bf 16}, 964--979.

\bibitem[Marchuk(1990)Marchuk]{Marchuk}
{\sc Marchuk, G.} (1990)
\newblock Splitting and alternating direction methods.
\newblock {\em Handbook of Numerical Analysis 1\/} (P.~Ciarlet \& J.~Lions
  eds). Handbook of Numerical Analysis, vol. 1.
\newblock Elsevier, pp. 197--462.

\bibitem[Mathew(2008)Mathew]{Mathew2008}
{\sc Mathew, T.} (2008)
\newblock {\em Domain Decomposition Methods for the Numerical Solution of
  Partial Differential Equations (Lecture Notes in Computational Science and
  Engineering)\/}, 1 edn.
\newblock Berlin: Springer.

\bibitem[Mathew {\em et~al.}(1998)Mathew, Polyakov, Russo, \& Wang]{Mathew1998}
{\sc Mathew, T.~P., Polyakov, P.~L., Russo, G. \& Wang, J.} (1998)
\newblock Domain decomposition operator splittings for the solution of
  parabolic equations.
\newblock {\em SIAM J. Sci. Comput.}, {\bf 19}, 912--932.

\bibitem[McLachlan \& Quispel(2002)McLachlan \& Quispel]{McLachlan}
{\sc McLachlan, R.~I. \& Quispel, G. R.~W.} (2002)
\newblock Splitting methods.
\newblock {\em Acta Numer.}, {\bf 11}, 341--434.

\bibitem[Pazy(1983)Pazy]{Pazy}
{\sc Pazy, A.} (1983)
\newblock {\em Semigroups of Linear Operators and Applications to Partial
  Differential Equations\/}. Applied Mathematical Sciences,  vol.~44.
\newblock New York: Springer.

\bibitem[Peaceman \& Rachford(1955)Peaceman \& Rachford]{Peaceman}
{\sc Peaceman, D.~W. \& Rachford, H.~H.} (1955)
\newblock The numerical solution of parabolic and elliptic differential
  equations.
\newblock {\em J. Soc. Indust. Appl. Math.}, {\bf 3}, 28--41.

\bibitem[Portero {\em et~al.}(2010)Portero, Arrarás, \& Jorge]{Portero2010}
{\sc Portero, L., Arrarás, A. \& Jorge, J.} (2010)
\newblock Contractivity of domain decomposition splitting methods for nonlinear
  parabolic problems.
\newblock {\em J Comput. Appl. Math.}, {\bf 234}, 1078--1087.

\bibitem[Quarteroni \& Valli(1999)Quarteroni \& Valli]{Quarteroni1999}
{\sc Quarteroni, A. \& Valli, A.} (1999)
\newblock {\em Domain Decomposition Methods for Partial Differential
  Equations\/}.
\newblock Numerical Mathematics and Scientific Computation.
\newblock Oxford: Clarendon.

\bibitem[Roub{\'i}{\v c}ek(2013)Roub{\'i}{\v c}ek]{Roubicek2013}
{\sc Roub{\'i}{\v c}ek, T.} (2013)
\newblock {\em Nonlinear Partial Differential Equations with Applications\/}.
  ISNM International Series of Numerical Mathematics,  vol. 153, second edn.
\newblock Basel: Birkh{\"{a}}user Verlag.

\bibitem[Samarskii {\em et~al.}(2013)Samarskii, Matus, \&
  Vabishchevich]{Samarskii2013}
{\sc Samarskii, A., Matus, P. \& Vabishchevich, P.} (2013)
\newblock {\em Difference Schemes with Operator Factors\/}. Mathematics and Its
  Applications,  vol. 546.
\newblock Dordrecht: Springer.

\bibitem[Swayne(1987)Swayne]{Swayne1987}
{\sc Swayne, D.~A.} (1987)
\newblock Time-dependent boundary and interior forcing in locally
  one-dimensional schemes.
\newblock {\em SIAM J. Sci. Statist. Comput.}, {\bf 8}, 755--767.

\bibitem[Temam(1979)Temam]{Temam1979}
{\sc Temam, R.} (1979)
\newblock {\em {N}avier--{S}tokes equations: Theory and Numerical Analysis\/}.
\newblock Studies in Mathematics and its Applications, revised edn.
\newblock Amsterdam: North Holland.

\bibitem[Temam(1997)Temam]{Temam1997}
{\sc Temam, R.} (1997)
\newblock {\em Infinite-Dimensional Dynamical Systems in Mechanics and
  Physics\/}.
\newblock Applied Mathematical Sciences, second edn.
\newblock New York: Springer.

\bibitem[Toselli \& Widlund(2005)Toselli \& Widlund]{Toselli2004}
{\sc Toselli, A. \& Widlund, O.} (2005)
\newblock {\em Domain Decomposition Methods -- Algorithms and Theory\/}.
  Springer Series in Computational Mathematics,  vol.~34.
\newblock Berlin; Heidelberg: Springer.

\bibitem[Triebel(1978)Triebel]{Triebel1978}
{\sc Triebel, H.} (1978)
\newblock {\em Interpolation Theory, Function Spaces, Differential
  Operators\/}.
\newblock Berlin: VEB Deutscher Verlag der Wissenschaften.

\bibitem[Vabishchevich(1989)Vabishchevich]{Vabishchevich1989}
{\sc Vabishchevich, P.} (1989)
\newblock Difference schemes with domain decomposition for solving
  non-stationary problems.
\newblock {\em USSR Comp. Math. Math.+\/}, {\bf 29}, 155--160.

\bibitem[Vabishchevich(2008)Vabishchevich]{Vabishchevich2008}
{\sc Vabishchevich, P.} (2008)
\newblock Domain decomposition methods with overlapping subdomains for the
  time-dependent problems of mathematical physics.
\newblock {\em Comput. Methods Appl. Math.}, {\bf 8}, 393--405.

\bibitem[Vabishchevich \& Zakharov(2013)Vabishchevich \&
  Zakharov]{Vabishchevich2013}
{\sc Vabishchevich, P. \& Zakharov, P.} (2013)
\newblock Domain decomposition scheme for first-order evolution equations with
  nonselfadjoint operators.
\newblock {\em Numerical Solution of Partial Differential Equations: Theory,
  Algorithms, and Their Applications\/} (O.~P. Iliev, S.~D. Margenov, P.~D.
  Minev, P.~S. Vassilevski \& L.~T. Zikatanov eds). Springer Proceedings in
  Mathematics \& Statistics, vol. 45.
\newblock New York: Springer, pp. 279--302.

\bibitem[Zeidler(1990)Zeidler]{Zeidler1990IIB}
{\sc Zeidler, E.} (1990)
\newblock {\em Nonlinear Functional Analysis and its Applications II/B,
  Nonlinear Monotone Operators\/}.
\newblock New York: Springer.

\end{thebibliography}

\end{document}